\documentclass[10pt, a4paper]{article}
\usepackage{amsfonts}
\usepackage{amssymb,amsthm, amsmath}
\usepackage{bbm}
\usepackage{xcolor}

\usepackage{url}

\newtheorem{theorem}{Theorem}
\newtheorem{definition}[theorem]{Definition}
\newtheorem{proposition}[theorem]{Proposition}
\newtheorem{corollary}[theorem]{Corollary}
\newtheorem{lemma}[theorem]{Lemma}

\theoremstyle{remark}

\newtheorem{example}[theorem]{Example}
\newtheorem{algorithm}[theorem]{Algorithm}
\newtheorem{remark}[theorem]{Remark}

\def\CaB{\mathcal{B}}

\def\CaL{\mathcal{L}}
\def\k{\mathbbmss{k}}
\def\N{\mathbb{N}}
\def\Z{\mathbb{Z}}
\def\C{\mathbb{C}}
\def\supp{\mathrm{supp}\, }
\def\rank{\mathrm{rank}\, }
\def\rowspan{\mathrm{rowspan}\, }
\def\rF{\mathrm{rF}}

\title{On decomposable semigroups and applications}

\author{J. I. Garc\'{\i}a-Garc\'{\i}a
\footnote{Departamento de Matem\'aticas, Universidad de C\'adiz,
E-11510 Puerto Real (C\'{a}diz, Spain). E-mail: ignacio.garcia@uca.es.
Partially supported by MTM2007-62346 and FQM-366.}
\\ M.A. Moreno-Fr\'{\i}as
\footnote{Departamento de Matem\'aticas, Universidad de C\'adiz,
E-11510 Puerto Real (C\'{a}diz, Spain). E-mail:
mariangeles.moreno@uca.es.
Partially supported by MTM2008-06201-C02-02 and FQM-298.}\\
A. Vigneron-Tenorio
\footnote{Departamento de Matem\'aticas, Universidad de C\'adiz,
E-11405 Jerez de la Frontera (C\'{a}diz, Spain). E-mail: alberto.vigneron@uca.es.
Partially supported by MTM2007-64704 and FQM-366.}}

\date{}

\begin{document}

\maketitle

\begin{abstract}

The aim of this work is to reduce the complexity of the available algorithms for computing the generator sets of a semigroup ideal  by using the Hermite normal form. In order to achieve it we introduce the concept of decomposable semigroup. If a semigroup is decomposable, the computation of its ideal is equivalent to compute the ideals of each semigroup in the decomposition, thus obtaining a reduction of the complexity of the algorithms. Furthermore, since these computations are mutually independent, they can be carried out in parallel. The concept of decomposable variety is introduced and a combinatorial characterization of decomposable semigroup is obtained. Some applications are also provided.


\smallskip
{\small \emph{MSC-class:} 13F20 (Primary), 15B36, 05E40, 13F55 (Secondary).}
\end{abstract}

\section*{Introduction}

Let $\k$ be a field. Given a finitely generated subsemigroup $S$ of an Abelian group and fixed one of its system of generators $A=\{a_1,\ldots ,a_n\}$, the semigroup ideal of $S$ is the binomial ideal (see \cite{Herzog70})
$$I_S=\langle X^\alpha -X^\beta | \sum_{i=1}^n \alpha_ia_i=\sum_{i=1}^n \beta_ia_i\rangle \subset \k[X_1,\ldots ,X_n]$$ with
$\k[X_1,\ldots ,X_n]$ the $S-$graduated polynomial ring where the $S$-degree $a_i$ is assigned to the indeterminate $X_i$. Thus, the $S$-degree of $X^\alpha =X_1^{\alpha_1}\cdots X_n^{\alpha_n}$ is $\sum _{i=1}^n \alpha _i a_i\in S$ (see \cite{Miller} for further details).

The study of semigroup ideals began in the last third of the 20th century (see \cite{Herzog70} and the references therein) and has become an important research area due to its connections with other scientific fields such as Algebraic Statistic, Coding Theory, Combinatorics, Integer Programming, Toric Geometry, etc. One of the most prominent topics in this research area is the improvement of computational aspects for checking ideal properties and for computing system of generators (see \cite{BCMP}, \cite{Hemmecke}, \cite{Sturmfels95} and \cite{Vigneron} and the references therein).

The main purpose of this work is to present a preprocessing method of polynomial complexity that improves the computation of the ideal for a decomposable semigroup. A semigroup $S=\langle A \rangle $ is {\em decomposable} if it is the direct sum of some proper subsemigroups $S_i=\langle A_i \rangle$ and then
\begin{equation}\label{def_ideal_descomp}I_S= I_{S_1}+ \cdots + I_{S_t},\end{equation} with $A=\sqcup _{i=1}^t A_i,$ otherwise it is irreducible. Note that each ideal $I_{S_i}$ is included in a polynomial ring with as many variables as elements belong to $A_i$  (strictly less than $n$). Theorem \ref{main_theorem} states that a semigroup is decomposable if and only if the Hermite normal form of the matrix associated to its system of generator has a special form. Therefore it can be detected whether a semigroup is decomposable and obtain its decomposition by using an algorithm of polynomial complexity (see \cite{hnf}). Since the complexity for computing a semigroup ideal is simply exponential in the number of variables (see \cite{Sturmfels93}), for decomposable semigroups the exponent of the complexity can be reduced to the maximum of the number of variables of $I_{S_i}$ with $i=1,\ldots,t.$
In summary, the complexity is reduced by using a preprocessing algorithm of polynomial complexity and once we know the decomposition of the semigroup, the computation of its generating set can be done in parallel to further reduce its computation time.

Decomposition (\ref{def_ideal_descomp}) allows us to take advantages of computation of some generating sets of $I_S.$ In this way, each type of generating set of $I_S$ (Gr\"{o}bner bases, Graver bases, universal
Gr\"{o}bner bases, Markov bases and universal Markov bases) can be obtained directly from the corresponding generating set of the ideals $I_{S_i}$ in (\ref{def_ideal_descomp}). In addition, some properties of decomposable semigroups and their ideals can be studied from the semigroups and ideals that appear in their decompositions. In this way, we prove that $I_S$ is a complete intersection if and only if $I_{S_i}$ is a complete intersection for every $i,$ and that a decomposable semigroup is a gluing if and only if at least one of the subsemigroups of its decomposition is a gluing.

The above purely algebraic decomposition has a geometric interpretation. Since every ideal defines an algebraic variety, decomposition (\ref{def_ideal_descomp}) can be used to obtain a decomposition of the variety and that allows us to introduce the concept of {\it decomposable variety} which are varieties that admit parametrizations with simple formulations.

If $S$ satisfies that $S\cap (-S)=\{0\},$ Theorem \ref{caracterizacion_combinatoria} gives a combinatorial characterization of decomposable semigroups by using the simplicial complex introduced in \cite{Eliahou} and the simplicial complex used in \cite{BCMP} (see \cite{OjVi} for further details).

The contents of this paper are organized as follows.
In Section \ref{preliminares} some definitions, notation and some known results are introduced.
In Section \ref{reduccion_matriz} it is showed how the Hermite normal
form can be used to obtain a diagonalization of a matrix which is
called HNF-diagonal matrix. Algorithm \ref{alg_matrix_fresi} is the
key that allows us to easily compute the decomposition of a semigroup.
In Section \ref{semigroup_fresi_decomposition}
decomposable semigroups are characterized by using the HNF-diagonalization of a matrix associated to $S.$
In Section \ref{sec_varieties} decomposable varieties are introduced and a method to obtain simple parametrizations is presented.
In Section \ref{ideal_descomponible} some relationships between decomposable semigroups (and their ideals) and the semigroups (and their ideals) appearing in the decomposition are showed.
The goal of Section
\ref{sec_caract_comb} is to obtain
a combinatorial characterization of
decomposable semigroups.
Finally in Section
\ref{ejemplo}, our results are illustrated with an example from
Algebraic Statistics.

\section{Preliminaries}\label{preliminares}

This work includes many and different mathematical objects, in this section we summarize some definitions and results that are useful for the understanding of this work.

Start by defining the elements we use to describe the decomposition of a semigroup.
Let $A=\left\{a_1,a_2,\ldots , a_n\right\}$ be the system of generators of the semigroup $S$.
Denote by $I_{\langle A'\rangle}$ the ideal of the semigroup $\langle A'\rangle$ where $A'=\{a_{i_1},\dots,a_{i_j}\}$ is a subset of $A.$
This ideal is included in the polynomial ring $\k [X_{i_1},\ldots,X_{i_j}]\subset \k[X_1,\ldots,X_n]$.
Denote by $\k [A']$ the subring of $\k [X_1,\ldots ,X_n]$ and its monomials by $X_{A'}$.

Given a lattice $\CaL\subset\Z^n,$ it has associated the binomial ideal  $\langle X^{\alpha ^+}-X^{\alpha ^-} | \alpha\in \CaL \rangle\subset \k[X_1,\ldots,X_n]$ where $\alpha ^+\in \Z_+^n$ and $\alpha ^-\in \Z_+^n$ are the unique vectors with disjoint supports such that $\alpha = \alpha ^+ -\alpha ^-$(see \cite{Sturmfels95}). It is well known that the set of the lattice ideals is equal to the set of semigroup ideals (see \cite{Vigneron}). In our case, if $\ker S$ is the lattice of integer solutions of the linear system $Ax=0,$ the lattice ideal $I_{\ker S}$ is equal to the semigroup ideal $I_S.$ Let  $\supp (\alpha)$ be the set $\{i|\, \alpha _i\neq 0 \}$ where $\alpha \in \N ^n.$
In the same way, $\supp (X ^\alpha)$ is $\supp (\alpha)$ and $\supp (X^{\alpha ^+}-X^{\alpha ^-})=\supp (X ^{\alpha^+})\cup \supp (X ^{\alpha^-}).$

In order to fix notation, for a decomposable semigroup $S=\langle A\rangle,$ define the decomposition of $S$ as the unique irreducible decomposition (up to permutations)
\begin{equation}\label{def_semig_descomp} S=\langle A_1\rangle\oplus \langle A_2\rangle \oplus \cdots \oplus \langle A_t\rangle\end{equation} where $\oplus$ is the direct sum of subsemigroups and $A=\sqcup _{i=1}^t A_i.$ If $S_i$ is the subsemigroup generated by $\langle A_i\rangle$ for $i=1,\dots, t,$ we obtain $I_S = I_{S_1}+\cdots +I_{S_t}.$ Considering above decomposition, the lattice $\ker S$ is equal to  $\ker (\langle A_1\sqcup A_2\sqcup \cdots \sqcup  A_t\rangle)$ up to a permutation of its coordinates. In order to denote this equality, up to permutations, we will use the symbol $\equiv$ (for instance $\ker S \equiv \ker (\langle A_1\sqcup A_2 \sqcup \cdots \sqcup A_t\rangle)$).

An interesting problem is the computation of minimal system of generators of semigroups $S$ with the property $S\cap (-S)=\{0\}.$ By Nakayama's lemma (see \cite{BCMP}),
every minimal generating sets of $I_S$ have the same cardinality. Denote by $Betti(S)$ the set of $S- $degrees of the elements of a minimal system of generators of $I_S$. It is known that if a semigroup $S$ is cancellative, $S\cap (-S)=\{0\}$ and it verifies  the condition of ascending chain (see \cite{Bullejos}), then $Betti(S)$ is the same for every minimal system of generators of $I_S$ and every system of generators of $I_S$ has the same number of binomials of $S$-degree $m\in S$ for all $m\in Betti(S)$. Any minimal system of generators of the ideal is known as a Markov basis.

Indispensable binomials (see \cite{Aoki}, \cite{Charalambous07}, \cite{GarciaOjeda} and \cite{OjVi2}) are needed to study the uniqueness of Markov bases, they are the binomials that belong (up to a scalar multiple) to every system of binomial generators of an ideal. There exists a unique Markov basis (up to a scalar multiple of its elements) of an ideal if and only if such ideal is generated by its indispensable binomials. In any case, to check if a Markov basis is unique  it is necessary to perform computations of high complexity in $n$ indeterminates.

Given the previous definitions, the ideal of a semigroup $I_S$ is a complete intersection if it is generated by $n-\rank (\ker S)$ binomials (see \cite{Fischer}). The study of these kind of ideals is classical in Commutative Algebra and Algebraic Geometry.

Other semigroups appearing in the literature are the
semigroups which are the gluing of two semigroups (see \cite{Rosales1}). A semigroup $S$ minimally generated by $C_1\sqcup C_2$
is the gluing of $S'=\langle C_1\rangle$ and $S''=\langle C_2\rangle,$ if there exists a set of generators of $I_S$ of the form
$\rho_1\cup \rho_2 \cup \{X^\alpha -X^\beta\},$ where $\rho_1,\rho_2$ are sets of generators of $I_{S'}$ and $I_{S''},$ and
$X^\alpha -X^\beta\in I_S$ with $X^\alpha\in \k [C_1]$ and $X^\beta
\in \k [C_2].$ Equivalently, there exists $d\in S'\cap
S''\setminus\{0\}$ such that $G(S')\cap G(S'')=G(\{d\}),$ where
$G(S'),$ $G(S'')$ and $G(\{d\})$ are the associated commutative
groups of $S',$ $S''$ and $\{d\}$ (see Theorem 1.4 in
\cite{Rosales1} for details). Note that $d$ is the $S$-degree of
$X^\alpha.$ This element $d\in S$ is called the gluing degree
associated to the partition, in this case $X^\alpha -X^\beta$ is
called a gluing binomial.

\section{HNF-diagonalization of matrices}\label{reduccion_matriz}

Given a matrix $B\in \Z ^{p\times q},$ denote by $F(B)\subset \Z ^q$ the set of rows of $B,$ and $\rF(B)=\rowspan _\Z(F(B)).$

\begin{definition}\cite[The "row version" of Definition 2.4.2]{Cohen}
We will say that an $m\times n$ matrix $M=(m_{ij})$ with integer
coefficients is in Hermite normal form (abbreviated HNF) if there
exists $r \leq m$ and a strictly increasing map $f$ from $[1,r]$
to $[1,n]$ satisfying the following properties:
\begin{enumerate}
 \item for $1\leq i \leq r$, $m_{i f(i)}\geq 1$, $m_{ij}=0$
 if $j< f(i)$ and $0 \leq m_{i f(k)} < m_{k f(k)}$ if $k> i$,
 \item the last $m-r$ rows of $M$ are equal to $0$.
\end{enumerate}
\end{definition}

Denote by $GL_n(\mathbb{Z})$ the group of matrices with integer
 coefficients which are invertible, i.e. whose determinant is equal
 to $\pm 1$.

\begin{theorem}\cite[Theorem 2.4.3]{Cohen}
 Let $L$ be an $m\times n$ matrix with coefficients in $\mathbb{Z}$.
 Then there exists a unique $m\times n$ matrix $H=(h_{i j})$ in HNF
 of the form $UL=H$ with  $U\in GL_n(\mathbb{Z}).$ In this case, we write \mbox{HNF}(L)=H.
\end{theorem}

For a given positive integer $n$ and $1\leq i,j\leq n, (i\neq j)$, consider $C_{i\leftrightarrow j}$, the matrix obtained from the
identity matrix $I_n$ interchanging the columns $i$ and $j$. These
matrices are known as column elementary matrices and they fulfill
$\det(C_{i\leftrightarrow j})=-1$ and $C_{i\leftrightarrow
j}=C_{i\leftrightarrow j}^{-1}$. Let $M$ be an $m\times n$ matrix
with integer coefficients. The matrix $MC_{i\leftrightarrow j}$ is
the matrix resulting from interchanging the columns $i$ and $j$ in
 $M$. In the sequel, we denote this matrix by
$Q$. This matrix is called {\it column permutation matrix}.
Similarly, {\it row permutation matrix} can be defined. Observe that
permutation matrices are nonsingular.

\begin{definition}\label{HNF-diagonal} An $m \times n$ matrix $D$ with integer coefficients is an
HNF-diagonal matrix if it satisfies the following conditions:
\begin{enumerate}
\item the null rows are located on the bottom side of the matrix,
\item the null columns are located  on the right side of the matrix,
\item every block  $D_i$ has maximal rank, it is in HNF and for each disjoint partition
of  $F(D_i),$ $B_1 \sqcup B_2,$ we have:
$(\bigcup_{f\in F(B_1)} \supp (f))\cap (
\cup_{f\in F(B_2)} \supp (f))\neq \emptyset.$
\end{enumerate}

That is, an HNF-diagonal matrix has the
following shape:
\begin{equation}\label{D}\left(\begin{array}{ccccccccc} D _1 & \Theta & \Theta & \cdots & \Theta
\\ \Theta & D_2 & \Theta & \cdots & \Theta \\
\Theta & \Theta & D_3 & \cdots & \Theta \\ \vdots & \vdots & \vdots
& \ddots & \vdots
\\  \Theta & \Theta & \Theta &  \cdots &   D _t
\end{array}\right),\end{equation}
where for  $i=1,\ldots ,t-1,$ $D_i$ is an HNF-matrix with non-zero
columns, $\Theta$ are the null matrices with the corresponding
orders and $D_t$ can be a null matrix or an HNF-matrix with zero rows
and/or columns.
\end{definition}

\begin{definition}\label{matrix_fresi} An $m\times n$ matrix
$L$ with integer coefficients is HNF-diagonalizable if there exists
a unimodular matrix P and a column permutation matrix $Q$ such that
$PLQ$ is an HNF-diagonal matrix.
\end{definition}

This is equivalent to say that an $m\times n$ matrix $L$ with
integer coefficients is HNF-diagonalizable if there exists a column
permutation
matrix $Q$ such that $\mbox{HNF}(LQ)$ is an HNF-diagonal matrix.\\

\begin{lemma}\label{HNF_equiv}
  Let $L$ be an $m\times n$ matrix with integer coefficients. Then $L$ is
  HNF-diagonalizable if and only if there exists a column permutation
  matrix
 $Q$ and an HNF-diagonal matrix  $D$  such that
$\rF(LQ)=\rF(D),$ and then
$\rF(L)\equiv \rF(LQ)= \rF(D)=\rF(D_1)\times \cdots \times \rF(D_t).$
\end{lemma}
\begin{proof}
 The proof follows from  {\it Test for Equality} (see \cite[p.74]{Cohen})
\end{proof}

Thus a matrix $L$ is HNF-diagonalizable if and only if the lattice $\rF(L)$ is (up to a permutation of its coordinates) the Cartesian product of some of its sublattices.

\begin{lemma}\label{aa}There exists an algorithm to determine if a matrix $L$
is HNF-diagonalizable.  If $L$ is HNF-diagonalizable this algorithm  finds two matrices $Q$ and
$P$ such that $PLQ$ is an HNF-diagonal matrix.
\end{lemma}

\begin{proof}
By Lemma \ref{HNF_equiv}, $L$ is HNF-diagonalizable if and only if there exists a column permutation
matrix $Q$ and an HNF-diagonal matrix  $D$  such that $\rF(LQ)=\rF(D)$.
Since  $Q^{-1}=Q$, then $\rF(L)=\rF(DQ^{-1})=\rF(DQ)$.
Therefore $\mbox{HNF}(L)=\mbox{HNF}(DQ)$. Trivially $\mbox{HNF}(DQ)$ is equal to $DQ$ up to a permutation of its rows.

In any case, the above condition can be detected by checking the supports of the rows of $\mbox{HNF}(L).$
\end{proof}

An algorithm satisfying Lemma \ref{aa} is now presented.

\begin{algorithm}\label{alg_matrix_fresi}{HNF-diagonalization}\\
{\em In:}  An $m\times n$ matrix $L$ with integer coefficients.\\
{\em Out:} This algorithm  detects  if  $L$ is HNF-diagonalizable
and finds two matrices  $Q$ and  $P$ satisfying  Definition
\ref{matrix_fresi}.
\begin{enumerate}
\item Compute $H$ the Hermite normal form of $L,$
    $$H=\left(\begin{array}{c}h_1\\h_2\\ \vdots \\ h_m \end{array}\right) \in \Z^{m\times n}.$$
    Let $P_2$ be a unimodular matrix such that  $P_2L=H.$

\item Initialize $\Lambda =\{2,\ldots ,m \},$ $\Lambda _j$ is the $j-$th element of $\Lambda$ and
 $B_i=\supp h_i,$ $F_i=\emptyset$ for every $i=1,\ldots ,m.$ Set $F_1=\{1\}$ and $i=1.$

\item \label{llamada1} For $j=1$ to card($\Lambda)$ (the cardinality of $\Lambda$) do:

\begin{itemize}
\item If $B_i\cap B_{\Lambda _j}\neq \emptyset,$ $B_i=B_i\cup B_{\Lambda _j},$ $F_i=F_i\cup \Lambda _j$ and
$\Lambda = \Lambda \setminus \Lambda _j.$
\item $j=j+1.$
\item If $\Lambda =\emptyset$ and $i=1$ then $L$ is not HNF-diagonalizable and the algorithm ends.
\item If $\Lambda =\emptyset$ and $i\neq 1$ then $L$ is HNF-diagonalizable and go to Step \ref{llamada2}.
Otherwise, $i=i+1$ and go to Step \ref{llamada1}.
\end{itemize}

\item \label{llamada2} Let $P_1$ be the row permutation matrix  to place the rows
of $H$ according to the sets  $F_j,$
$$M=P_1H=\left(\begin{array}{c}h_{F_{1,1}}\\ \vdots\\ h_{F_{1,{\rm card}(F_1)}} \\ h_{F_{2,1}}\\ \vdots\\
h_{F_{2,{\rm card}(F_2)}}\\ \vdots \\ h_{F_{i,1}}\\ \vdots\\
h_{F_{i,{\rm card}(F_i)}} \end{array}\right) \in \Z^{m\times
n},$$where $F_{j,t}$ is the $t-$th element in $F_j,$ and let $Q$ be
the column permutation matrix such that $MQ$ is the HNF-diagonal
matrix. Observe that the matrix  $Q$ is determined by the sets
$B_j$ modified in Step 3 and the zero-columns.

\item The matrices $Q$ and $P=P_1P_2$ are the matrices we are looking for.
\end{enumerate}

\end{algorithm}

Given a matrix $L,$ from now on assume that its associated
HNF-diagonal matrix  has the same form as (\ref{D}).

\section{Decomposable semigroups}\label{semigroup_fresi_decomposition}

In this section decomposable semigroups are characterized by
means of the HNF-diagonalization of a matrix which rows are a
base of the lattice $\ker S.$ We present a lemma that is a
useful tool for this characterization.

\begin{lemma}\label{lemma_ideal}
$S=\langle A\rangle$ is a decomposable semigroup if and only if
$\ker S\equiv  \ker S_1\times \cdots \times \ker S_t.$
\end{lemma}

\begin{proof}

Assume that $S=\langle A\rangle$ is a decomposable semigroup as in (\ref{def_semig_descomp}) and let $\alpha$ be an element belongs to $\ker S.$ Then $\alpha$ can be expressed as the vector $(\alpha _1,\ldots ,\alpha_t)$ where each $\alpha_i$ corresponds to the coordinates of $\alpha$ related to $A_i$ for $i=1,\ldots ,t,$ (i.e. $\alpha \equiv (\alpha _1,\ldots ,\alpha_t)$).

Suppose now without loss of generality that $\alpha _1\notin \ker S_1.$ Hence $A_1\alpha_1\neq 0$ and $A_1\alpha_1^+ \neq A_1\alpha_1^-.$ For any $i,$ take $n_i=A_i\alpha_i^+\in S$ and  $m_i=A_i\alpha_i^-\in S.$ So the element $m=A \alpha ^+$ can be written as $n_1+\cdots +n_t$ and $m_1+\cdots +m_t$ with $n_1\neq m_1.$ In this case, one obtain that $S$ is not the direct sum $S=S_1\oplus \cdots \oplus S_t.$ Then $\alpha _1\in \ker S_1$ and $\ker S\equiv  \ker S_1\times \cdots \times \ker S_t.$

Conversely, let $\ker S$ equivalent to $\ker S_1\times \cdots \times \ker S_t.$ Trivially, $S=S_1 + \cdots + S_t.$ So we only need to prove that for any $m\in S$ such that $m=n_1+\cdots +n_t = m_1+\cdots +m_t$ with $n_i$ and $m_i$ belong to $S_i$ for all $i,$ each $n_i$ must be equal to $m_i.$ As $n_i,m_i\in S_i,$ there exist $\alpha_i$ and $\beta_i$ nonnegative integer vectors such that $n_i=A_i\alpha_i$ and $m_i=A_i\beta_i.$ Therefore $m=\sum_i A_i\alpha _i =\sum _i A_i \beta _i$ and $(\alpha_1-\beta_1,\ldots ,\alpha_t-\beta_t)\in \ker S_1 \times \cdots \times \ker S_t \equiv \ker S.$ Then $\alpha _i-\beta _i$ belongs to $\ker S_i$ and $n_i=m_i$ for each $i.$

\end{proof}

The aim of the following result is the  characterization of
decomposable semigroups by using HNF-diagonalization.

\begin{theorem}\label{main_theorem}
Let $S$ be a semigroup and $L$ be a matrix such that $\ker S=\rF(L)$. Then $S$ is a decomposable semigroup if and only if $L$ is an HNF-diagonalizable
matrix.
\end{theorem}

\begin{proof}
Suppose that  $S$ is a decomposable semigroup. By Lemma \ref{lemma_ideal},
$\ker S\equiv  \ker S_1 \times \cdots \times \ker S_t.$
Consider $L_i$ a base of each $\ker S_i,$ where $L_i$ is in HNF. Let  $L'$ be the matrix
$$L'=\left(\begin{array}{ccccccccc} L _1 & \Theta & \Theta & \cdots & \Theta
\\ \Theta & L_2 & \Theta & \cdots & \Theta \\
\Theta & \Theta & L_3 & \cdots & \Theta \\ \vdots & \vdots & \vdots & \ddots & \vdots
\\  \Theta & \Theta & \Theta &  \cdots &   L _t \\ \Theta & \Theta & \Theta &  \cdots &   \Theta
\end{array}\right)\in \Z ^{m\times n}.$$
Since $L$ and $L'$ span equivalent lattices, we have that there exists a column permutation
matrix $Q$ such that $\rF(LQ)= \rF(L')$ and $\mbox{HNF}(LQ)=\mbox{HNF}(L')$
(by Lemma \ref{HNF_equiv}). It is easy to see that $\mbox{HNF}(L')$
is an HNF-diagonal matrix, therefore  $L$ is HNF-diagonalizable.

Conversely, suppose that  $L$ is HNF-diagonalizable. In this
case, if  $D$ is the  HNF-diagonal matrix associated to  $LQ,$
$\rF(LQ)=\rF(D)= \rF(D_1) \times \cdots \times \rF(D_t),$ where  $D_i$ represents the blocks of the matrix  $D.$ We conclude that  $S$ is decomposable and the subsemigroups $S_i$ of
its irreducible decomposition are the semigroups generated by the
column of  $A$ indicated by the elements of the corresponding blocks  $D_i$.
In this case,
$I_S = I_{\rF(LQ)}= I_{\rF(D_1)} + \cdots + I_{\rF(D_t)}$ with $I_{S_i}=I_{\rF(D_i)}.$
\end{proof}

As described in the proof, given an HNF-diagonalization of
$L$ the irreducible decomposition $S=\langle A_1\rangle\oplus \cdots \oplus \langle A_t\rangle$ is obtained
where for every $j$, $A_j$ is the set formed
by the elements of $A$ corresponding to the blocks $D_j$ of the associated
HNF-diagonal matrix.

\begin{remark}\label{nota_libre_torsion}
If the semigroup $S$ is torsion free we can study the decomposition of the semigroup $S$ by applying Algorithm \ref{alg_matrix_fresi} to the matrix whose columns are formed by a system of generators of $S$ and in this case the decomposition can be obtained operating directly in $S$. The examples of Section \ref{sec_varieties}, Section \ref{ejemplo} and Table \ref{tabla} have been done applying directly the decomposition over their corresponding semigroups.
\end{remark}

\begin{remark}\label{nota_aspecto_agradable}
Algorithm \ref{alg_matrix_fresi} gives a method to obtain the minimal decomposition of a semigroup $S$ (in case this decomposition exists), but also it can be seen as an algorithm to obtain a simple system generators of a semigroup isomorphic to $S.$ For any given torsion free semigroup, a simple system of generators can be obtained by using Remark \ref{nota_libre_torsion}. If the semigroup is not torsion free, that system of generators can be obtain from the lattice $\ker S$ by applying \cite[Chapter 2]{Rosales3}.
\end{remark}

\section{Decomposable varieties}\label{sec_varieties}

It is well known (see \cite[Chapter 4]{Sturmfels95}) that every torsion free semigroup $S=\langle a_1,\ldots ,a_n\rangle\subset \Z ^m$ has an affine variety $V(I_S)$ determined by the set of zeros of its ideal $I_S\subset \k[X_1,\ldots ,X_n].$ This variety is parametrized by the equations $x_i=l^{a_i},$ with $i=1,\ldots,n$ and $l\in (\k^*)^m$.

If one consider the decomposable semigroup (\ref{def_semig_descomp}), $V(I_S)=V(I_{S_1})\cap \cdots \cap V(I_{S_t}),$ where all the varieties $V(I_{S_i})\subset \k ^{\mbox{{\tiny card}} (A_i)}$ are embedded in $\k^n$. We say that a variety is a decomposable variety if it is obtained from a decomposable semigroup.

Clearly the above parametrization depends on the generators of the semigroup $S$, althought the variety $V(I_S)$ is determined by the ideal $I_S$ (the variety depends on the relations of the generators of $S$). Thus, systems of generators of isomorphic semigroups represent the same variety.

From Remark \ref{nota_aspecto_agradable} we obtain not only a decomposition when such decomposition exists, we can determine a simple reparametrization.

\begin{example}
Consider the toric variety determined by the parametrization
$$\left\{\begin{array}{ccllllll}
x_1 & = & l_1^{-1} & l_2^{-4} & l_3^{7} & l_4^{7} & l_5^{-1}\\
x_2 & = & l_1^{-4} & l_2^{-4} & l_3^{8} & l_4^{16} & l_5^{-4}\\
x_3 & = & l_1^{12} & l_2^{-9} & l_3^{12} & l_4^{6} & l_5^{-3}\\
x_4 & = & l_1^{4} & l_2^{-3} & l_3^{4} & l_4^{2} & l_5^{-1}\\
x_5 & = & l_1^{-3} &  & l_3 & l_4^{9} & l_5^{-3}\\
x_6 & = & l_1^{-2} & l_2^{-8} & l_3^{14} & l_4^{14} & l_5^{-2}\\
x_7 & = & l_1^{8} & l_2^{-6} & l_3^{8} & l_4^{4} & l_5^{-2}\\
\end{array}\right.,$$
and let $S\subset \Z^5$ the semigroup associated to the parametrization generated by the columns of the matrix
$$A=\left( \begin {array}{ccccccc} -1&-4&12&4&-3&-2&8\\\noalign{\medskip}
-4&-4&-9&-3&0&-8&-6\\\noalign{\medskip}7&8&12&4&1&14&8
\\\noalign{\medskip}7&16&6&2&9&14&4\\\noalign{\medskip}-1&-4&-3&-1&-3&
-2&-2\end {array} \right).$$

Now by applying Algorithm \ref{alg_matrix_fresi} and Remark \ref{nota_libre_torsion}, the semigroup $S$ is the direct sum of
$$\left\langle \left( \begin {array}{cccc} -1&-4&-3&-2\\\noalign{\medskip}-4&-4&0&-8
\\\noalign{\medskip}7&8&1&14\\\noalign{\medskip}7&16&9&14
\\\noalign{\medskip}-1&-4&-3&-2\end {array} \right)\right\rangle \bigoplus \left\langle\left( \begin {array}{ccc} 12&4&8\\\noalign{\medskip}-9&-3&-6
\\\noalign{\medskip}12&4&8\\\noalign{\medskip}6&2&4
\\\noalign{\medskip}-3&-1&-2\end {array} \right)\right\rangle,
$$
and it is isomorphic (it has the same ideal up to permutation of the variables) to the direct sum of the semigroups $S_1$ and $S_2$,
$$\left\langle \left( \begin {array}{cccc} 1&0&-1&2\\0&4&
4&0\\0&0&0&0\\0&0&0&0\\0&0&0&0\end {array} \right)\right\rangle \bigoplus \left\langle \left( \begin {array}{ccc} 0&0&0\\ 0&0&0\\ 3&1&2\\ 0&0&0\\0&0&0\end {array} \right)\right\rangle.$$

Using this decomposition a simple reparametrization of the affine toric variety $V$ is
$$\left\{\begin{array}{ccllllll}
x_1 & = & q_1\\
x_2 & = &  & q_2^{4}\\
x_3 & = &  &  & q_3^{3}\\
x_4 & = &  & & q_3 \\
x_5 & = & q_1^{-1} & q_2^{4}\\
x_6 & = & q_1^{2}\\
x_7 & = &  & & q_3^{2} \\
\end{array}\right.$$
Furthermore the ideal $I_S$ is the direct sum of the ideals $I_{S_1}\subset \C[x_1,x_2,x_5,x_6]$ and
$I_{S_2}\subset \C[x_3,x_4,x_7],$
$I_S=\left\langle
x_1^2-x_6,\,
x_1x_5-x_2
\right\rangle+ \left\langle x_4^2-x_7,\,  x_4^3-x_3\right\rangle$ and
$V=V\left(\left\langle
x_1^2-x_6,\,
x_1x_5-x_2 \right\rangle\right) \cap V\left(\left\langle x_4^2-x_7,\,  x_4^3-x_3 \right\rangle\right).$
\end{example}

\section{Some applications of the decomposition}\label{ideal_descomponible}

From (\ref{def_semig_descomp}), it is possible to characterize some properties of the
ideal $I_S$ using the decomposition $I_S= I_{S_1}+ \cdots
+ I_{S_t}.$ From now
on, we consider the notation fixed in
(\ref{def_semig_descomp}).

\begin{proposition}\label{main_corollary} Let
 $I_S = I_{S_1}+ \cdots + I_{S_t}$ be the decomposition of $I_S$.
 Then $I_S$ is generated by the disjoint union  of the
 generators of each  $I_{S_i}.$
\end{proposition}

\begin{proof} Trivial, using that the generator system of each
 $I_{S_i}$ do not have  common variables.
\end{proof}

\begin{remark}\label{markov_descomp}The above result is also true
for: Markov bases, universal Markov bases, Gr\"{o}bner bases, Graver bases and universal
Gr\"{o}bner bases, and  $Betti(S)=\bigcup _{i=1}^t Betti(S_i).$
\end{remark}

Under the assumption $S\cap (-S)=\{0\},$
the following result allows to
study the uniqueness of Markov basis of $I_S$ using the Markov bases
of the ideals  $I_{S_i}.$

\begin{corollary}Let
 $I_S= I_{S_1}+ \cdots + I_{S_t}$   be the decomposition of $I_S$.
 Then
 $I_S$ is generated by indispensable binomials
 ($I_S$ has  a unique Markov basis up to scalar multiple of its elements) if and only if
 $I_{S_i}$ is generated by indispensable binomials, for all  $i=1,\ldots , t.$
\end{corollary}

We show how to determine if the ideal $I_S$
is a complete intersection by using the ideals $I_{S_i}.$

\begin{corollary}Let
 $I_S= I_{S_1}+ \cdots + I_{S_t}$ be the decomposition of $I_S.$ $I_S$ is  a complete intersection if and only if $I_{S_i}$ is a complete intersection for every $i=1,\ldots ,t.$
\end{corollary}

Assuming again $S\cap (-S)=\{0\},$ the following result establishes that the {\em gluing} property
can be characterized by the decomposition of the semigroup.

\begin{proposition}
$S$ is a gluing if and only if there exists $i\in\{1,\ldots, t\}$ such that  $S_i$ is a gluing of semigroups.
\end{proposition}

\begin{proof}
Suppose that  $S$ is the gluing of  $S'$ and $S''.$ In this case,
there exists a  gluing binomial $X^\alpha -X^\beta \in I_S$ verifying
that its  $S$-degree  $d\in S$ is in  $S'\cap S''.$ Moreover, by
Lemma \ref{lemma_ideal},
$X^\alpha -X^\beta=X_{A_1}^{\alpha _1}X_{A_2}^{\alpha _2}\cdots
X_{A_t}^{\alpha _t} - X_{A_1}^{\beta _1}X_{A_2}^{\beta  _2}\cdots
X_{A_t}^{\beta_t},$ where $X_{A_i}^{\alpha _i} - X_{A_i}^{\beta
_i}\in I_{S_i}\subset I_S$ for  $i=1,\ldots ,t.$

If $d_i\in S_i$ is the $S$-degree of $X_{A_i}^{\alpha
_i},$ we have  $d=d_1+\cdots +d_t.$ Without loss of generality, suppose that  $d_1\neq 0.$

Clearly,  $d_1\in (S_1\cap S') \cap (S_1\cap S'')$ and  $d_1\in
G(S_1\cap S') \cap G(S_1\cap S'')\subset G(S') \cap G(S'')=d\Z .$ Since
 $S\cap (-S)= \{0\}$ and $d=d_1+\cdots +d_t,$ we have  $d_1=d.$
 Therefore,  $d_i=0$ for every  $i > 1,$ and  $d=d_1\in S_1.$ That is, the gluing
degree associated to the partition  $S=S'+S''$ belongs to a
semigroup in the decomposition of $S$.

As $G(S_1\cap S') \cap G(S_1\cap S'')=d_1\Z$ and $d_1\in
(S_1\cap S') \cap (S_1\cap S''),$ the semigroup $S_1$ is the gluing of $S_1\cap S'$ and $S_1\cap S''$.

Conversely, if any  $S_i$ is the gluing of  $S'_i$ and $S''_i$, then,
it is easy to prove that $S$ is the gluing of  $S'_i$ and  $S''_i\cup
\left(\bigcup _{j\neq i} S_j\right).$
\end{proof}

\section{Combinatorial results}\label{sec_caract_comb}

In this section, we consider that the semigroup $S$ satisfies $S\cap (-S)=\{0\}.$ If a Markov basis $\CaB$ of a semigroup ideal $I_S$ verifies that it
can be decomposed into two proper and disjoint subsets $\CaB_1$ and
$\CaB_2$ such that $\CaB=\CaB_1\sqcup \CaB_2$ and
$(\cup
_{f\in \CaB_1} \supp (f)) \cap (\cup _{f\in \CaB_2}
\supp (f)) =\emptyset,$
then the semigroup $S$ is
decomposable. A not necessarily irreducible decomposition is
determined by the generators of the semigroup given by the sets
$\bigcup _{f\in \CaB_1} \supp (f)$ and $\bigcup _{f\in \CaB_2} \supp
(f).$ The irreducible decomposition can be obtained applying this
idea to the subsets $\CaB_1$ and $\CaB_2.$ Thus the
decomposition of a semigroup is characterized by the decomposition
into subsets with disjoint support of a Markov basis of its
associated ideal. However, Markov bases can be characterized in a combinatorial way
through the study of some simplicial complexes, such as the complex
$\Delta _m$  (see \cite{BCMP}) or the complex $\nabla _m$ (see \cite{OjVi}), where
$\Delta_m = \{F \subseteq \{1,\ldots ,n\} |\, m - \sum_{i \in F} a_i
\in S \},$
$$\nabla_m = \{ F
\subseteq  \{\mbox{monomials in }\k[X_1,\ldots,X_n]\mbox{ of }S\mbox{-degree}\,\, m\}|\, \gcd(F) \neq 1 \}$$ and $\gcd(F)$ denotes the
greatest common divisor of the monomials in $F.$

A Markov basis of the semigroup ideal has an element
of $S$-degree $m\in S$ if, and only if, $\Delta_m$ (respectively
$\nabla_m$) are not connected. Furthermore, the elements of a given
degree, $m\in Betti(S),$ that appear in a Markov basis are fixed by the connected
components of such complexes and the binomials of this degree can be obtained from $\nabla_m.$ For every non-connected complex $\nabla_m$, the number of binomials obtained is equal to the number of connected components minus one (see \cite{BCMP} and \cite{OjVi} for
further details), therefore we obtain the following combinatorial characterization of decomposable semigroups.

\begin{theorem}\label{caracterizacion_combinatoria}
Let $S$ be a semigroup. The following statements are equivalent:
\begin{enumerate}
\item $S$ is decomposable.
\item There exist $C_1,C_2$ proper and disjoint subsets of $\{a_1,\ldots ,a_n\}$ fulfilling:
    \begin{itemize}
        \item $\{a_1,\ldots ,a_n\}= C_1\sqcup C_2.$
        \item For all $m \in S$ with $\nabla _m$ not connected, all the vertices of $\nabla _m$ belong to $\k[C_1]$ or $\k[C_2],$
        \item There exist $m_1,m_2\in S$ with $\nabla _{m_1}$ and $\nabla _{m_2}$ not connected such that the vertices of $\nabla _{m_1}$ belong to  $\k[C_1]$ and the vertices of $\nabla _{m_2}$ belong to $\k[C_2].$
    \end{itemize}
\item There exist $C_1,C_2$ proper and disjoint subsets  of $\{1,\ldots ,n\}$ satisfying:
    \begin{itemize}
        \item $\{1,\ldots ,n\}= C_1\sqcup C_2.$
        \item For all $m \in S$ with $\Delta _m$ not connected, all vertices of $\Delta _m$ belong to $C_1$ or $C_2,$
        \item There exist $m_1,m_2\in S$ with $\Delta _{m_1}$ and $\Delta _{m_2}$ not connected such that the vertices of $\Delta _{m_1}$ belong to  $C_1$ and the vertices of $\Delta _{m_2}$ belong to $C_2.$
    \end{itemize}
\end{enumerate}
\end{theorem}

\begin{proof}
Since the facets of $\Delta_m$ are the union of the supports of the monomials of the facets
of $\nabla_m$, {\em 2} and {\em 3}
are equivalent.

Now let us prove {\em 1} implies {\em 2}. Assume that $S$ is
decomposable. In such case, a Markov basis of $I_S$ is the union of
the Markov bases of the ideals $I_{S_i}$. With the notation fixed in
(\ref{def_semig_descomp}), consider $C_1=A_1$ and $C_2=\bigsqcup
_{i=2}^t A_i.$ Let $m$ be the $S$-degree of any binomial of a
Markov basis of any $I_{S_i}$ (without lost of generality we
consider $i=1$). In this case $\nabla_m$ is non-connected and it has
at least two different connected components, each of them with a
monomial in $\k[C_1]$.

If there exists $X_{C_1}^\alpha X_{C_2}^\beta\in \k [C_1,C_2]$ in $\nabla_m$, then there exists $X_{C_1}^\alpha X_{C_2}^\beta-X_{C_1}^\gamma\in I_S$ with $S$-degree $m$. Therefore the element $(\alpha - \gamma,\beta)$ belong to $\ker S_1 \times \cdots \times \ker S_t$. Hence $\beta\in \ker S_2 \times \cdots \times \ker S_t$, that is not possible since $S\cap (-S)=\{0\}.$ This concludes {\em 1} implies {\em 2}.

For {\em 2} implies {\em 1}, consider the sets $C_1$
and $C_2$ given in the statement. In this case, since the Markov
bases of $I_S$ are calculated from the connected components of the
complexes $\nabla_m$ and such complexes are contained in $\k[C_1]$
or $\k[C_2]$, these bases are formed exclusively by binomials in $\k
[C_1]$ and $\k[C_2]$. By hypotheses, these bases have elements of
both types.

Clearly $I_S=I_{\langle C_1 \rangle} + I_{\langle C_2 \rangle}$ and a Gr\"{o}bner basis of $I_S$ is formed by the union of one Gr\"{o}bner basis of $I_{\langle C_1 \rangle}$ and another  of $I_{\langle C_2 \rangle}.$

It is straightforward to prove that $S= \langle C_1 \rangle +
\langle C_2 \rangle.$ It only remains to prove that this sum is a
direct sum. If not, there would exist $m'\in \langle C_1 \rangle
\cap \langle C_2 \rangle$ and a binomial $X_{C_1}^\alpha -
X_{C_2}^\beta \in I_S$ with $S$-degree $m'$. Since the normal form
of this binomial with respect to a Gr\"{o}bner basis of $I_S$ is not
equal to zero, this is not possible.
\end{proof}

\section{Decomposable semigroups in some Statistical Models}\label{ejemplo}

One of the most important objects in Algebraic Statistics (see \cite{Diaconis}, \cite{Drton}, \cite{Gibilisco} and \cite{Pistone} about Algebraic Statistics) is the ideal associated to semigroups related with statistical models. Among the properties of this ideal, the uniqueness of its Markov bases is also a very studied problem  in this research area (see \cite{Aoki} and \cite{Takemura}). In this context, we  use our algorithm to check how many statistical models of the database \cite{DataBase} have associated a decomposable semigroup, and we also study the uniqueness of their Markov bases (see Table \ref{tabla}).

In order to illustrate the results of this work, we study the statistical model {\em "The binary graphical model of the bipyramid graph"}
(example {\verb"BPg_bin"} in \cite{DataBase}). The semigroup associated to this model is determined by the columns of the matrix
{\tiny $$A=\left(
\begin{array}{@{\hspace{3pt}}c@{\hspace{3pt}}c@{\hspace{3pt}}c@{\hspace{3pt}}c@{\hspace{3pt}}c@{\hspace{3pt}}c@{\hspace{3pt}}c@{\hspace{3pt}}c@{\hspace{3pt}}c@{\hspace{3pt}}c@{\hspace{3pt}}c@{\hspace{3pt}}c@{\hspace{3pt}}c@{\hspace{3pt}}c@{\hspace{3pt}}c@{\hspace{3pt}}c@{\hspace{3pt}}c@{\hspace{3pt}}c@{\hspace{3pt}}c@{\hspace{3pt}}c@{\hspace{3pt}}c@{\hspace{3pt}}c@{\hspace{3pt}}c@{\hspace{3pt}}c@{\hspace{3pt}}c@{\hspace{3pt}}c@{\hspace{3pt}}c@{\hspace{3pt}}c@{\hspace{3pt}}c@{\hspace{3pt}}c@{\hspace{3pt}}c@{\hspace{3pt}}c@{\hspace{3pt}}}
 1 & 0 & 0 & 0 & 0 & 0 & 0 & 0 & 0 & 0 & 0 & 0 & 0 & 0 & 0 & 0 & 1 & 0 & 0 & 0 & 0 & 0 & 0 & 0 & 0 & 0 & 0 & 0 & 0 & 0 & 0 & 0 \\
 0 & 1 & 0 & 0 & 0 & 0 & 0 & 0 & 0 & 0 & 0 & 0 & 0 & 0 & 0 & 0 & 0 & 1 & 0 & 0 & 0 & 0 & 0 & 0 & 0 & 0 & 0 & 0 & 0 & 0 & 0 & 0 \\
 0 & 0 & 1 & 0 & 0 & 0 & 0 & 0 & 0 & 0 & 0 & 0 & 0 & 0 & 0 & 0 & 0 & 0 & 1 & 0 & 0 & 0 & 0 & 0 & 0 & 0 & 0 & 0 & 0 & 0 & 0 & 0 \\
 0 & 0 & 0 & 1 & 0 & 0 & 0 & 0 & 0 & 0 & 0 & 0 & 0 & 0 & 0 & 0 & 0 & 0 & 0 & 1 & 0 & 0 & 0 & 0 & 0 & 0 & 0 & 0 & 0 & 0 & 0 & 0 \\
 0 & 0 & 0 & 0 & 1 & 0 & 0 & 0 & 0 & 0 & 0 & 0 & 0 & 0 & 0 & 0 & 0 & 0 & 0 & 0 & 1 & 0 & 0 & 0 & 0 & 0 & 0 & 0 & 0 & 0 & 0 & 0 \\
 0 & 0 & 0 & 0 & 0 & 1 & 0 & 0 & 0 & 0 & 0 & 0 & 0 & 0 & 0 & 0 & 0 & 0 & 0 & 0 & 0 & 1 & 0 & 0 & 0 & 0 & 0 & 0 & 0 & 0 & 0 & 0 \\
 0 & 0 & 0 & 0 & 0 & 0 & 1 & 0 & 0 & 0 & 0 & 0 & 0 & 0 & 0 & 0 & 0 & 0 & 0 & 0 & 0 & 0 & 1 & 0 & 0 & 0 & 0 & 0 & 0 & 0 & 0 & 0 \\
 0 & 0 & 0 & 0 & 0 & 0 & 0 & 1 & 0 & 0 & 0 & 0 & 0 & 0 & 0 & 0 & 0 & 0 & 0 & 0 & 0 & 0 & 0 & 1 & 0 & 0 & 0 & 0 & 0 & 0 & 0 & 0 \\
 0 & 0 & 0 & 0 & 0 & 0 & 0 & 0 & 1 & 0 & 0 & 0 & 0 & 0 & 0 & 0 & 0 & 0 & 0 & 0 & 0 & 0 & 0 & 0 & 1 & 0 & 0 & 0 & 0 & 0 & 0 & 0 \\
 0 & 0 & 0 & 0 & 0 & 0 & 0 & 0 & 0 & 1 & 0 & 0 & 0 & 0 & 0 & 0 & 0 & 0 & 0 & 0 & 0 & 0 & 0 & 0 & 0 & 1 & 0 & 0 & 0 & 0 & 0 & 0 \\
 0 & 0 & 0 & 0 & 0 & 0 & 0 & 0 & 0 & 0 & 1 & 0 & 0 & 0 & 0 & 0 & 0 & 0 & 0 & 0 & 0 & 0 & 0 & 0 & 0 & 0 & 1 & 0 & 0 & 0 & 0 & 0 \\
 0 & 0 & 0 & 0 & 0 & 0 & 0 & 0 & 0 & 0 & 0 & 1 & 0 & 0 & 0 & 0 & 0 & 0 & 0 & 0 & 0 & 0 & 0 & 0 & 0 & 0 & 0 & 1 & 0 & 0 & 0 & 0 \\
 0 & 0 & 0 & 0 & 0 & 0 & 0 & 0 & 0 & 0 & 0 & 0 & 1 & 0 & 0 & 0 & 0 & 0 & 0 & 0 & 0 & 0 & 0 & 0 & 0 & 0 & 0 & 0 & 1 & 0 & 0 & 0 \\
 0 & 0 & 0 & 0 & 0 & 0 & 0 & 0 & 0 & 0 & 0 & 0 & 0 & 1 & 0 & 0 & 0 & 0 & 0 & 0 & 0 & 0 & 0 & 0 & 0 & 0 & 0 & 0 & 0 & 1 & 0 & 0 \\
 0 & 0 & 0 & 0 & 0 & 0 & 0 & 0 & 0 & 0 & 0 & 0 & 0 & 0 & 1 & 0 & 0 & 0 & 0 & 0 & 0 & 0 & 0 & 0 & 0 & 0 & 0 & 0 & 0 & 0 & 1 & 0 \\
 0 & 0 & 0 & 0 & 0 & 0 & 0 & 0 & 0 & 0 & 0 & 0 & 0 & 0 & 0 & 1 & 0 & 0 & 0 & 0 & 0 & 0 & 0 & 0 & 0 & 0 & 0 & 0 & 0 & 0 & 0 & 1 \\
 1 & 0 & 0 & 0 & 0 & 0 & 0 & 0 & 1 & 0 & 0 & 0 & 0 & 0 & 0 & 0 & 0 & 0 & 0 & 0 & 0 & 0 & 0 & 0 & 0 & 0 & 0 & 0 & 0 & 0 & 0 & 0 \\
 0 & 1 & 0 & 0 & 0 & 0 & 0 & 0 & 0 & 1 & 0 & 0 & 0 & 0 & 0 & 0 & 0 & 0 & 0 & 0 & 0 & 0 & 0 & 0 & 0 & 0 & 0 & 0 & 0 & 0 & 0 & 0 \\
 0 & 0 & 1 & 0 & 0 & 0 & 0 & 0 & 0 & 0 & 1 & 0 & 0 & 0 & 0 & 0 & 0 & 0 & 0 & 0 & 0 & 0 & 0 & 0 & 0 & 0 & 0 & 0 & 0 & 0 & 0 & 0 \\
 0 & 0 & 0 & 1 & 0 & 0 & 0 & 0 & 0 & 0 & 0 & 1 & 0 & 0 & 0 & 0 & 0 & 0 & 0 & 0 & 0 & 0 & 0 & 0 & 0 & 0 & 0 & 0 & 0 & 0 & 0 & 0 \\
 0 & 0 & 0 & 0 & 1 & 0 & 0 & 0 & 0 & 0 & 0 & 0 & 1 & 0 & 0 & 0 & 0 & 0 & 0 & 0 & 0 & 0 & 0 & 0 & 0 & 0 & 0 & 0 & 0 & 0 & 0 & 0 \\
 0 & 0 & 0 & 0 & 0 & 1 & 0 & 0 & 0 & 0 & 0 & 0 & 0 & 1 & 0 & 0 & 0 & 0 & 0 & 0 & 0 & 0 & 0 & 0 & 0 & 0 & 0 & 0 & 0 & 0 & 0 & 0 \\
 0 & 0 & 0 & 0 & 0 & 0 & 1 & 0 & 0 & 0 & 0 & 0 & 0 & 0 & 1 & 0 & 0 & 0 & 0 & 0 & 0 & 0 & 0 & 0 & 0 & 0 & 0 & 0 & 0 & 0 & 0 & 0 \\
 0 & 0 & 0 & 0 & 0 & 0 & 0 & 1 & 0 & 0 & 0 & 0 & 0 & 0 & 0 & 1 & 0 & 0 & 0 & 0 & 0 & 0 & 0 & 0 & 0 & 0 & 0 & 0 & 0 & 0 & 0 & 0 \\
 0 & 0 & 0 & 0 & 0 & 0 & 0 & 0 & 0 & 0 & 0 & 0 & 0 & 0 & 0 & 0 & 1 & 0 & 0 & 0 & 0 & 0 & 0 & 0 & 1 & 0 & 0 & 0 & 0 & 0 & 0 & 0 \\
 0 & 0 & 0 & 0 & 0 & 0 & 0 & 0 & 0 & 0 & 0 & 0 & 0 & 0 & 0 & 0 & 0 & 1 & 0 & 0 & 0 & 0 & 0 & 0 & 0 & 1 & 0 & 0 & 0 & 0 & 0 & 0 \\
 0 & 0 & 0 & 0 & 0 & 0 & 0 & 0 & 0 & 0 & 0 & 0 & 0 & 0 & 0 & 0 & 0 & 0 & 1 & 0 & 0 & 0 & 0 & 0 & 0 & 0 & 1 & 0 & 0 & 0 & 0 & 0 \\
 0 & 0 & 0 & 0 & 0 & 0 & 0 & 0 & 0 & 0 & 0 & 0 & 0 & 0 & 0 & 0 & 0 & 0 & 0 & 1 & 0 & 0 & 0 & 0 & 0 & 0 & 0 & 1 & 0 & 0 & 0 & 0 \\
 0 & 0 & 0 & 0 & 0 & 0 & 0 & 0 & 0 & 0 & 0 & 0 & 0 & 0 & 0 & 0 & 0 & 0 & 0 & 0 & 1 & 0 & 0 & 0 & 0 & 0 & 0 & 0 & 1 & 0 & 0 & 0 \\
 0 & 0 & 0 & 0 & 0 & 0 & 0 & 0 & 0 & 0 & 0 & 0 & 0 & 0 & 0 & 0 & 0 & 0 & 0 & 0 & 0 & 1 & 0 & 0 & 0 & 0 & 0 & 0 & 0 & 1 & 0 & 0 \\
 0 & 0 & 0 & 0 & 0 & 0 & 0 & 0 & 0 & 0 & 0 & 0 & 0 & 0 & 0 & 0 & 0 & 0 & 0 & 0 & 0 & 0 & 1 & 0 & 0 & 0 & 0 & 0 & 0 & 0 & 1 & 0 \\
 0 & 0 & 0 & 0 & 0 & 0 & 0 & 0 & 0 & 0 & 0 & 0 & 0 & 0 & 0 & 0 & 0 & 0 & 0 & 0 & 0 & 0 & 0 & 1 & 0 & 0 & 0 & 0 & 0 & 0 & 0 & 1
\end{array}
\right).$$}
Note that the ideal $I_{\langle A \rangle}$ belongs to a polynomial ring in
$32$ indeterminates. Applying Algorithm \ref{alg_matrix_fresi} to the matrix $A$ (see Remark \ref{nota_libre_torsion}), we get
 $A$ is HNF-diagonalizable and its associated HNF-diagonal matrix
is
$$D=\left(\begin{array}{cccccccc} D_1 & \Theta & \Theta & \Theta & \Theta & \Theta & \Theta & \Theta \\
 \Theta & D_1 & \Theta & \Theta & \Theta & \Theta & \Theta & \Theta \\
\Theta & \Theta & D_1 & \Theta & \Theta & \Theta & \Theta & \Theta \\
\Theta & \Theta & \Theta & D_1 & \Theta & \Theta & \Theta & \Theta \\
\Theta & \Theta & \Theta & \Theta & D_1 & \Theta & \Theta & \Theta \\
\Theta & \Theta & \Theta & \Theta & \Theta & D_1 &\Theta & \Theta \\
\Theta & \Theta & \Theta & \Theta & \Theta & \Theta & D_1 & \Theta \\
\Theta  & \Theta & \Theta & \Theta & \Theta & \Theta & \Theta & D_1 \\
\Theta  & \Theta & \Theta & \Theta & \Theta & \Theta & \Theta & \Theta \\
\end{array}\right),$$ where $D_1=\left( \begin {array}{cccc} 1&0&0&-1\\\noalign{\medskip}0&1&0&1
\\\noalign{\medskip}0&0&1&1\end {array} \right)$, and the permutation matrix $Q$
swaps the columns of  $A$ grouping them in the following sets:
$$A_1=\{1,9,17,25\},\, A_2=\{2,10,18,26\},\, A_3=\{3,11,19,27\},\, A_4=\{4,12,20,28\},\,$$ $$A_5=\{5,13,21,29\},\,
A_6=\{6,14,22,30\},\, A_7=\{7,15,23,31\},\, A_8=\{8,16,24,32\}.$$ In this example, we identify the index set $A_i$ with the columns of $A.$ Once the decomposition has been done,  it only remains to do a fast computation
to obtain that
 the Markov basis of  $I_{\langle A_1 \rangle}=I_{\langle D_1 \rangle}\subset \k[x_1,x_9,x_{17},x_{25}]$ is $\{x_1x_{25}-x_9x_{17}\}$,
 this basis is unique,
 $I_{\langle A_1 \rangle}$ is a complete intersection
 and
 $\langle D_1\rangle$ is the gluing of the semigroups $S'_1=\langle(1,0,0),(-1,1,1)\rangle$ and $S'_2=\langle(0,1,0),(0,0,1)\rangle$. Then $\langle A_1 \rangle$ is the gluing of the semigroups generated by the columns $\{1,25\}$ and $\{9,17\}$ of $A.$ Therefore, by the results presented in this paper, the ideal $I_{\langle A \rangle}\subset \k [x_1,\ldots ,x_{32}]$
 has a unique Markov basis,
$\{x_1x_{25}-x_9x_{17},\,  x_2x_{26}-x_{10}x_{18},\,  x_3x_{27}-x_{11}x_{19},\,  x_{4}x_{28}-x_{12}x_{20},\,  x_{5}x_{29}-x_{13}x_{21},
x_{6}x_{30}-x_{14}x_{22},\,  x_{7}x_{31}-x_{15}x_{23},\,  x_{8}x_{32}-x_{16}x_{24}\},$
and  it is a complete intersection. Furthermore, the semigroup $S=\langle A\rangle$ is the gluing of two semigroups and it decomposes as the direct sum
of eight subsemigroups $S=\langle A_1\rangle\oplus \cdots \oplus \langle A_8\rangle.$

Note that instead of performing highly complex computations in the ring of polynomials in $32$ indeterminates, we perform
the computations in rings in four indeterminates. In the example and given the particular nature of the matrix $D$ we need only to calculate a
unique ideal.

In Table \ref{tabla}, using \cite[Corollary 14]{OjVi2},  we complete the study on the uniqueness of the Markov bases for the examples that are decomposable in the database \cite{DataBase}.
 In addition, we collect the number of semigroups in the irreducible decomposition and the number of
generators of each semigroup of the decomposition.
\begin{table}[h]
  \centering
    \begin{tabular}{||c|c|c|c||}
\hline
\hline
          & {\bf N$^o$ of semigroups} & {\bf N$^o$ of generators} & {\bf Unique} \\
\hline
\hline
    {\bf BM4r2-2\_bin} & 2     & 8--8  & no \\ \hline
    {\bf G15g\_bin} & 2     & 8--8  & YES \\ \hline
    {\bf G17g\_bin} & 4     & 4--4--4--4 & YES \\ \hline
    {\bf 5-4m1\_bin} & 2     & 16--16 & yes \\ \hline
    {\bf BM5r3-2\_bin} & 2     & 16--16 & no \\ \hline
    {\bf BM5r3-4\_bin} & 2     & 16--16 & no \\ \hline
    {\bf BPg\_bin} & 8     & 4-4-4-4-4-4-4-4 & YES \\ \hline
    {\bf G34g\_bin} & 2     & 16--16 & NO \\ \hline
    {\bf G40g\_bin} & 2     & 16--16 & NO \\ \hline
    {\bf G42g\_bin} & 2     & 16--16 & YES \\ \hline
    {\bf G45g\_bin} & 2     & 16--16 & YES \\ \hline
    {\bf G46g\_bin} & 4     & 8-8-8-8 & no \\ \hline
    {\bf G47g\_bin} & 2     & 16-16 & YES \\ \hline
    {\bf SPg\_bin} & 2     & 16--16 & YES \\ \hline
    {\bf 6-1I2\_bin} & 4     & 16-16-16-16 & yes \\ \hline
    {\bf G111g\_bin} & 2     & 32--32 & NO \\ \hline
    {\bf G117g\_bin} & 2     & 32--32 & NO \\ \hline
    {\bf G133g\_bin} & 2     & 32--32 & NO \\ \hline
    {\bf G135g\_bin} & 2     & 32--32 & NO \\ \hline
    {\bf G136g\_bin} & 2     & 32--32 & NO \\ \hline
    {\bf G144g\_bin} & 2     & 32--32 & NO \\ \hline
    {\bf G156g\_bin} & 2     & 32--32 & NO \\ \hline
    {\bf G158g\_bin} & 2     & 32--32 & NO \\ \hline
    {\bf G161g\_bin} & 4     & 16-16-16-16 & NO \\ \hline
    {\bf G162g\_bin} & 2     & 32--32 & NO \\ \hline
    {\bf G164g\_bin} & 2     & 32--32 & NO \\ \hline
    {\bf G165g\_bin} & 2     & 32--32 & YES \\ \hline
    {\bf G92g\_bin} & 2     & 32--32 & NO  \\ \hline \hline
    \end{tabular}
    \caption{Decomposable semigroups of \cite{DataBase}. Last column represents the uniqueness property. The normal letters "no/yes" correspond to the already solved cases in \cite{DataBase} and the capital ones "NO/YES" correspond to the solved cases in this work.}\label{tabla}
\end{table}
It is noteworthy that using the decomposition obtained in Table
\ref{tabla}, we lowered the time for computing the Markov bases of the ideals by more that 50\% using the program \verb"4ti2" (see \cite{4ti2}) for computing the Gr\"{o}bner bases.

\end{document}